\documentclass[11pt,oneside,leqno]{amsart}
\usepackage{amsxtra}
\usepackage{amsopn}
\usepackage{color}
\usepackage{amsmath,amsthm,amssymb}
\usepackage{amscd}
\usepackage{amsfonts}
\usepackage{latexsym}
\usepackage{verbatim}
\usepackage{pb-diagram}
\usepackage{amssymb}
\usepackage{tikz}

\usepackage{booktabs}
\usepackage{multirow}
\usepackage{array}
\usepackage[all]{xy}

\theoremstyle{plain}
\newtheorem{theorem}{Theorem}[section]
\newtheorem*{theorem*}{Theorem}

\newtheorem*{prop*}{Proposition}

\newtheorem*{cor*}{Corollary}
\newtheorem{rem}[theorem]{Remark}

\newtheorem*{mt*}{Main Theorem}
\newcommand{\rr}{\mathbb{R}}
\newcommand{\C}{\mathbb{C}}
\newcommand{\del}{\partial}
\newcommand{\delbar}{\overline{\del}}
\newcommand\Z{{\mathbb Z}}

\begin{document}
\title[Cohomological aspects on complex and symplectic manifolds]{Cohomological aspects on complex and symplectic manifolds}
\author{Nicoletta Tardini}
\date{\today}
\address{Dipartimento di Matematica\\
Universit\`{a} di Pisa \\
Largo Bruno Pontecorvo 5, 56127 \\
Pisa, Italy}
\email{tardini@mail.dm.unipi.it}
\thanks{Partially supported by GNSAGA
of INdAM}
\begin{abstract}
We discuss how quantitative cohomological informations could provide
qualitative properties on complex and symplectic manifolds.
In particular we focus on the Bott-Chern and the Aeppli cohomology groups in both cases, since they represent useful tools in studying non K\"ahler geometry. We give an overview on the comparisons among the dimensions of the cohomology groups that can be defined and we show how we reach the $\del\delbar$-lemma in complex geometry and the Hard-Lefschetz condition in symplectic geometry.
For more details we refer to \cite{angella-tardini} and \cite{tardini-tomassini-symplectic}.
\end{abstract}

\maketitle

\section{Introduction}
\label{sec:1}
In this note we discuss the informations that we can obtain
on both complex and symplectic (not necessarily K\"ahler) manifolds
studying the space of forms endowed with suitable differential operators;
in particular, we focus on how quantitative cohomological properties
could provide qualitative informations on the manifold.
Recall that a smooth K\"ahler manifold is a complex manifold endowed with a Hermitian metric whose fundamental $2$-form is $d$-closed. For dimensional reasons
every Riemann surface is K\"ahler but in higher dimension this is not true in general.
In complex dimension two, K\"ahlerness can be topologically characterized in terms of the first Betti number (see \cite{kodaira}, \cite{miyaoka}, \cite{siu}) but a similar result does not hold in dimension
greater than two. Nevertheless there are many topological obstructions to the existence of a K\"ahler metric on a manifold, for example the odd Betti numbers are even and the even Betti numbers are positive.
These results follow from the strong requests on the involved geometric structures and their deep relations. It seems therefore natural to ask what happens
if we weaken those structures and/or their relations.
In particular we could weaken the complex condition looking at 
non integrable almost-complex structures or we could look at complex manifolds
with a weaker metric condition (e.g., balanced metrics \cite{michelsohn},
SKT metrics \cite{bismut}, etc).
On the other side we could ignore the (almost-)complex structure focusing
the attention on the existence of a non-degenerate $d$-closed $2$-form (i.e., a symplectic form) moving therefore to symplectic geometry.
In any case, an important global tool in studying smooth manifold is furnished by cohomology, more precisely cohomology groups that are invariant for the considered geometric structures.\\
In complex non-K\"ahler geometry it turns out that the classical de Rham and Dolbeault cohomology groups do not suffice in studying a complex manifold (see e.g., \cite{angella-iwasawa}), indeed many informations are contained
in the Bott-Chern and Aeppli cohomologies, defined, on a complex manifold $X$, respectively as
$$
H^{\bullet,\bullet}_{BC}(X):=\frac{Ker\; \partial\cap Ker\; \overline\partial}{Im \;\partial\overline\partial}\;,\qquad
H^{\bullet,\bullet}_{A}(X):=\frac{Ker \;\partial\overline\partial}{Im\; \partial+Im\;\overline\partial}.
$$
They represent a bridge between
a topological invariant (the de Rham cohomology) and a complex invariant
(the Dolbeault cohomology).
In general we have the following picture:
$$\xymatrix{
  & H^{\bullet,\bullet}_{BC}(X) \ar[d]\ar[ld]\ar[rd] & \\
  H^{\bullet,\bullet}_{\partial}(X) \ar[rd] & H^{\bullet}_{dR}(X;\mathbb{C}) \ar[d] & H^{\bullet,\bullet}_{\overline\partial}(X) \ar[ld] \\
  & {\phantom{\;.}} H^{\bullet,\bullet}_{A}(X) \;, &
} $$
where the maps are the ones induced by the identity.
Generally such maps are neither injective nor surjective but when the map
$H^{\bullet,\bullet}_{BC}(X)\longrightarrow H^{\bullet,\bullet}_A(X)$
is injective, the manifold $X$ is said to satisfy the \emph{$\del\delbar$-lemma}.
Every K\"ahler manifold satisfies the $\del\delbar$-lemma but the converse is not true.
In this paper we will compare the dimensions of these cohomology groups recalling some results contained in \cite{angella-tomassini-inequality} and \cite{angella-tardini}; in particular we will focus on how
just knowing the dimensions of the Bott-Chern (and dually Aeppli) cohomology groups we can understand whether the $\del\delbar$-lemma holds.
More precisely,
\begin{theorem}[see Theorem \ref{angella-tomassini-disuguaglianza}, Theorem \ref{characterization deldelbarlemma}, Remark \ref{dimensione2}]
 Let $X$ be a compact complex manifold. Then,
 the following facts are equivalent:
\begin{enumerate}
\item{the $\del\delbar$-lemma holds on $X$;}
\medskip
\item{$\Delta^k:=\sum_{p+q=k}
\left( dim_\mathbb{C} H^{p,q}_{BC}(X) + dim_\mathbb{C} H^{p,q}_{A}(X) \right) - 2\, b_k \;=\; 0 \;$, for any $k\in\mathbb{Z}$;}
 \medskip
 \item{$\sum_{p+q=k} \left( dim_\mathbb{C} H^{p,q}_{BC}(X) - dim_\mathbb{C} H^{p,q}_{A}(X) \right) \;=\; 0$, for any $k\in\mathbb{Z}$.}
\end{enumerate}
 Moreover, if $X$ has complex dimension $2$, then
 $X$ is K\"ahler if and only if $\Delta^2=0$.
\end{theorem}
In a similar fashion on a compact symplectic manifold $(X,\omega)$ it is possible
to consider the symplectic Bott-Chern and Aeppli cohomology groups, as defined by
Tseng and Yau in \cite{tsengyauI} by using the operators $d$ and its
symplectic-adjoint $d^\Lambda$. They are the appropriate cohomology groups in order to study symplectic Hodge theory.
In the present work, similarly to the complex case, we will consider
some comparisons among the dimensions of these cohomology groups
collecting some results contained in \cite{angella-tomassini-algebraic},
\cite{angella-tardini} and \cite{tardini-tomassini-symplectic}.
It turns out that the symplectic Bott-Chern cohomology $H^\bullet_{d+d^\Lambda}(X)$ (and dually Aeppli cohomology $H^\bullet_{dd^\Lambda}(X)$)
suffices to characterize the $dd^\Lambda$-lemma,
indeed we have the following
\begin{theorem}[see \cite{brylinski}, \cite{mathieu}, \cite{merkulov},
\cite{yan}, \cite{cavalcanti}, \cite{angella-tomassini-algebraic}, Theorem \ref{characterization HLC}]
Let $(X,\omega)$ be a compact symplectic manifold of dimension $2n$. Then, the following facts are equivalent:
\begin{enumerate}
\item{the \emph{hard-Lefschetz condition} (HLC for short) holds, i.e., the maps
\[
\left[\omega\right]^k:H^{n-k}_{dR}(X)\longrightarrow H^{n+k}_{dR}(X), \qquad 0\leq k\leq n
\]
are all isomorphisms;}
\medskip
\item{the \emph{$dd^\Lambda$-lemma} holds, i.e., the natural maps induced by the identity $H^{\bullet}_{d+d^\Lambda}(X)
\longrightarrow H^{\bullet}_{dR}(X)$ are injective;}
\medskip
\item{$\Delta^k:=dim \;H^k_{d+d^\Lambda}(X) + dim\; H^k_{dd^\Lambda}(X)  - 2\, b_k\;=\; 0 \;$, for any $k\in\mathbb{Z}$.}
\end{enumerate}
Moreover, if $X$ has dimension $4$, then
 $X$ satisfies $HLC$ if and only if $\Delta^2=0$.
\end{theorem}
 
\noindent {\em Acknowledgements.}
The author would like to thank the organizers Daniele Angella, Paolo de Bartolomeis, Costantino Medori and Adriano Tomassini for the kind invitation
to the \emph{INdAM Meeting Complex and Symplectic Geometry} held at Palazzone of Cortona. Many thanks to all the participants at the conference who contributed to create such a nice environment.
Special thanks also to Adriano Tomassini and Daniele Angella for their constant support, encouragement, for many useful discussions on the subject
and for their contribution to the results obtained jointly with the author.
This proceeding is dedicated to the memory of the very kind person and excellent mathematician Paolo de Bartolomeis.\\

\smallskip

\section{Complex cohomologies}
\label{sec:2}

We begin this section with some preliminaries and fixing some notations.
Let $X$ be a compact complex manifold of complex dimension $n$. With $A^{p,q}(X)$ we denote the space of complex $(p,q)$-forms on $X$. As a consequence of the integrability of the complex structure
the triple $\left(A^{\bullet,\bullet}(X),\,\partial,\,\overline\partial\right)$ represents a double complex, indeed the following relations hold: $\partial^2=0$, $\overline\partial^2=0$ and
$\partial\overline\partial+\overline\partial\partial=0$.\\
The complex \emph{de Rham}, \emph{Dolbeault} and \emph{conjugate Dolbeault}
cohomology groups of $X$ have been widely studied and they are defined, respectively, as
$$
H^\bullet_{dR}(X;\mathbb{C}):=\frac{Ker \;d}{Im \;d}\;,\qquad
H^{\bullet,\bullet}_{\overline\partial}(X):=\frac{Ker\; \overline\partial}{Im\;\overline\partial}\;,\qquad
H^{\bullet,\bullet}_{\partial}(X):=\frac{Ker\; \partial}{Im\; \partial}\;.
$$
Roughly speaking, if we draw a double complex as follows, for the Dolbeault cohomology we are
looking at vertical arrows, since the operator $\delbar$ changes the second degree of a $(p,q)$-form, and for its conjugate we are looking at horizontal arrows,
since the operator $\del$ changes the first degree of a $(p,q)$-form.
For a more detailed explanation of the interpretation of a double complex as a sum of indecomposable objects as zig-zag, dots and squares we refer to \cite{angella-survey}.

\begin{center}
\begin{tikzpicture}
\newcommand\un{1.3}

\draw[help lines, step=\un] (0,0) grid (3*\un,3*\un);

\foreach \x in {0,...,2}
  \node at (\un*.5+\un*\x,-.3) {\x};
\foreach \y in {0,...,2}
  \node at (-.3,\un*.5+\un*\y) {\y};

\coordinate (E) at (1*\un+1/2*\un, 2*\un+1/2*\un);
\coordinate (N) at (2*\un+1/2*\un, 2*\un+1/2*\un);
\coordinate (U) at (0*\un+1/2*\un, 1*\un+1/2*\un);
\coordinate (V) at (0*\un+1/2*\un, 2*\un+1/2*\un);

\newcommand{\raggio}{1*\un pt}
\fill (E) circle (\raggio);
\fill (N) circle (\raggio);
\fill (U) circle (\raggio);
\fill (V) circle (\raggio);

\draw[->] (E) -- (N);
\draw[->] (U) -- (V);

\begingroup\makeatletter\def\f@size{6}\check@mathfonts
\node at (1*\un+1/2*\un+.1, 2*\un+1/2*\un-.3) {$\gamma$};
\node at (2*\un+1/2*\un+.1, 2*\un+1/2*\un-.3) {$\delta$};
\node at (0*\un+1/2*\un+.1, 1*\un+1/2*\un-.3) {$\alpha$};
\node at (0*\un+1/2*\un+.1, 2*\un+1/2*\un+.3) {$\beta$};
\endgroup

\end{tikzpicture}
\end{center}

\medskip

For instance, in the above picture we mean that
$\delbar\alpha=\beta$, $\del\alpha=\del\beta=0$, $\del\gamma=\delta$ and
$\delbar\gamma=\delbar\delta=0$. So $\alpha$ and $\beta$ are representatives
of two non-trivial classes in $H_\del^{\bullet,\bullet}(X)$ and $\beta$ represents
the trivial class in $H_{\delbar}^{0,2}(X)$.
Similarly goes for $\delta$ and $\gamma$.
Notice that we can not have two consecutive vertical (resp. horizontal) arrows
because $\delbar^2=0$ (resp. $\del^2=0$).\\
Nevertheless there is no natural map between the de Rham (a topological invariant) and Dolbeault (a holomorphic invariant) cohomologies, in this sense a bridge between them is furnished by
the \emph{Bott-Chern} \cite{bott-chern} and the \emph{Aeppli} \cite{aeppli} cohomology groups defined by
$$
H^{\bullet,\bullet}_{BC}(X):=\frac{Ker\; \partial\cap Ker\; \overline\partial}{Im \;\partial\overline\partial}\;,\qquad
H^{\bullet,\bullet}_{A}(X):=\frac{Ker \;\partial\overline\partial}{Im\; \partial+Im\;\overline\partial}.
$$
The same definitions can be stated, more generally, for a double complex
$(B^{\bullet,\bullet},\partial,\overline\partial)$ of vector spaces.
In this way we are taking into accounts the corners in the
double complex of forms. For example looking at this picture

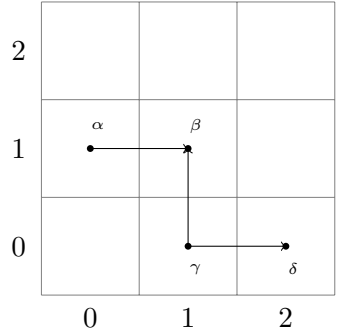
\begin{center}
\begin{tikzpicture}
\newcommand\un{1.3}

\draw[help lines, step=\un] (0,0) grid (3*\un,3*\un);

\foreach \x in {0,...,2}
  \node at (\un*.5+\un*\x,-.3) {\x};
\foreach \y in {0,...,2}
  \node at (-.3,\un*.5+\un*\y) {\y};

\coordinate (A) at (0*\un+1/2*\un, 1*\un+1/2*\un);
\coordinate (B) at (1*\un+1/2*\un, 1*\un+1/2*\un);
\coordinate (C) at (1*\un+1/2*\un, 0*\un+1/2*\un);
\coordinate (D) at (2*\un+1/2*\un, 0*\un+1/2*\un);

\newcommand{\raggio}{1*\un pt}
\fill (A) circle (\raggio);
\fill (B) circle (\raggio);
\fill (C) circle (\raggio);
\fill (D) circle (\raggio);

\draw[->] (A) -- (B);
\draw [->] (C)-- (B);
\draw [->] (C)-- (D);

\begingroup\makeatletter\def\f@size{6}\check@mathfonts

\node at (0*\un+1/2*\un+.1, 1*\un+1/2*\un+.3) {$\alpha$};
\node at (1*\un+1/2*\un+.1, 1*\un+1/2*\un+.3) {$\beta$};
\node at (1*\un+1/2*\un+.1, 0*\un+1/2*\un-.3) {$\gamma$};
\node at (2*\un+1/2*\un+.1, 0*\un+1/2*\un-.3) {$\delta$};

\endgroup

\end{tikzpicture}
\end{center}

the forms $\alpha$ and $\gamma$ are representatives of two non-trivial classes in 
$H_{A}^{\bullet,\bullet}(X)$ and $\beta,\;\delta$ in $H_{BC}^{\bullet,\bullet}(X)$.
Namely, ingoing corners contribute to the Bott-Chern cohomology and
outgoing corners to the Aeppli cohomology.\\
As regards the algebraic structure, a very easy computation shows that
the product induced by the wedge product on forms induces a structure of algebra for the Bott-Chern cohomology of a complex manifold
$H^{\bullet,\bullet}_{BC}(X)$ and a structure of $H^{\bullet,\bullet}_{BC}(X)$-module for the Aeppli cohomology
$H^{\bullet,\bullet}_{A}(X)$.\\
In \cite{schweitzer}, see also \cite{kodaira-spencer}, Hodge theory for the Bott-Chern and the Aeppli cohomologies is developed. In particular, once fixed a Hermitian metric
$g$ on $X$ the Bott-Chern and the Aeppli cohomology groups of $X$ are, respectively, isomorphic to the kernel of the following $4^{th}$-order elliptic self-adjoint differential operators
$$ 
\Delta_{BC}^g \;:=\;
\left(\del\delbar\right)\left(\del\delbar\right)^*+\left(\del\delbar\right)^*\left(\del\delbar\right)+
\left(\delbar^*\del\right)\left(\delbar^*\del\right)^*+\left(\delbar^*\del\right)^*\left(\delbar^*\del\right)+\delbar^*\delbar+\del^*\del $$
and
$$\Delta_{A}^g \;:=\; \del\del^*+\delbar\delbar^*+\left(\del\delbar\right)^*\left(\del\delbar\right)+\left(\del\delbar\right)\left(\del\delbar\right)^*+\left(\delbar\del^*\right)^*\left(\delbar\del^*\right)+\left(\delbar\del^*\right)\left(\delbar\del^*\right)^*\,.
$$
Therefore these cohomologies
are finite-dimensional vector spaces. Moreover, differently
from the Poincar\'e and Serre duality for the Dolbeault cohomology,
the Hermitian duality does not
preserve these cohomologies; more precisely when a Hermitian metric is fixed on $X$, the $\mathbb{C}$-anti-linear Hodge-$*$-operator induces an (un-natural) isomorphism between the Bott-Chern cohomology and the Aeppli cohomology, namely
$$
*:H^{p,q}_{BC}(X)\longrightarrow H^{n-p,n-q}_{A}(X)
$$
is an isomorphism for any $p,q\in\mathbb{Z}$; this means that we do not have symmetry with respect to the center in the Bott-Chern (and Aeppli) diamond.
Therefore, we have the following equalities:
$\dim_{\C} H^{p,q}_{BC}(X)=\dim_{\C} H^{q,p}_{BC}(X)=
\dim_{\C} H^{n-q,n-p}_{A}(X)=\dim_{\C} H^{n-p,n-q}_{A}(X)$,
where the first one and the last one are due to the fact that the conjugation preserves
the Bott-Chern and the Aeppli cohomologies respectively (giving a symmetry in the Bott-Chern diamond with respect to the central column).

\begin{rem}
Notice that, in general, the isomorphism $H^{\bullet,\bullet}_{BC}(X)
\simeq Ker\Delta_{BC}^g$ is of vector spaces not algebras,
indeed the wedge product of harmonic forms is not necessarily harmonic.
The study of Hermitian metrics whose space of Bott-Chern harmonic forms
has a structure of algebra has been developed in \cite{angella-tomassini-formality}
and in \cite{tardini-tomassini-formality}
in terms of geometric formality.
\end{rem}

By definition, the identity induces natural maps of (bi-)graded vector spaces between the Bott-Chern, Dolbeault, de Rham, and Aeppli cohomologies:
$$\xymatrix{
  & H^{\bullet,\bullet}_{BC}(X) \ar[d]\ar[ld]\ar[rd] & \\
  H^{\bullet,\bullet}_{\partial}(X) \ar[rd] & H^{\bullet}_{dR}(X;\mathbb{C}) \ar[d] & H^{\bullet,\bullet}_{\overline\partial}(X) \ar[ld] \\
  & {\phantom{\;.}} H^{\bullet,\bullet}_{A}(X) \;. &
} $$
Recall that a compact complex manifold is said to satisfy the {\em $\partial\overline\partial$-Lemma} if the natural map $H^{\bullet,\bullet}_{BC}(X)\longrightarrow H^{\bullet,\bullet}_{A}(X)$ is injective. This is equivalent to any of the above maps being an isomorphism, see \cite[Lemma 5.15]{deligne-griffiths-morgan-sullivan}.
Since any compact K\"ahler manifold satisfies the $\partial\overline\partial$-lemma
the Bott-Chern and Aeppli cohomologies could provide more informations on a compact complex manifold which does not admit any K\"ahler metric.
For this reason, from now on, we will implicitly assume that our manifolds are not K\"ahler.

\subsection{Inequalities on compact complex manifolds}

In this section we are mainly interested in discussing quantitative cohomological informations on complex manifolds with the final aim of understanding which
integers can appear as dimensions of cohomology groups of complex manifolds.
In the compact K\"ahler case the Hodge decomposition Theorem states that the
Dolbeault cohomology groups give a decomposition of the de Rham cohomology,
inducing at the level of cohomology the decomposition of complex forms
in $(p,q)$-forms.
This is no longer true if we drop the K\"ahler assumption.
Fr\"ohlicher in \cite{frolicher} construct a spectral sequence whose first page
is isomorpic to the Dolbeault cohomology and converging to the de Rham
cohomology proving, consequently, 
that on any compact 
complex manifold $X$ there is a topological lower bound for the Hodge numbers
(the dimensions of the Dolbeault cohomology groups) in terms of
the Betti numbers (the dimensions of the de Rham cohomology groups), namely
for any $k\in\mathbb{Z}$,
$$ \sum_{p+q=k} \dim_\mathbb{C} H^{p,q}_{\overline\partial}(X)  \;\geq\; b_k \;. $$

A Fr\"ohlicher type inequality has been proven by Angella and Tomassini in
\cite{angella-tomassini-inequality} taking into consideration the Bott-Chern and the Aeppli cohomology groups. For clearness we report here the complete statement.

\begin{theorem}[{\cite[Theorem A, Theorem B]{angella-tomassini-inequality}}]\label{angella-tomassini-disuguaglianza}
 Let $X$ be a compact complex manifold. Then, for any $k\in\mathbb{Z}$,
 $$ \Delta^k(X) \;:=\; \sum_{p+q=k} \left( \dim_\mathbb{C} H^{p,q}_{BC}(X) + \dim_\mathbb{C} H^{p,q}_{A}(X) \right) - 2\, b_k \;\geq\; 0 \;. $$
 Moreover, $X$ satisfies the $\partial\overline\partial$-Lemma if and only if, for any $k\in\Z$, there holds $\Delta^k(X)=0$.
\end{theorem}

It provides a lower bound for the dimension of the Bott-Chern and Aeppli cohomologies in terms of the Betti numbers
(the proof actually shows a lower bound also in terms of the Hodge numbers), and it yields also a 
quantitative characterization of the $\partial\overline\partial$-Lemma.
The proof of this Theorem is essentially algebraic and it is based on Varouchas
exact sequences \cite{varouchas}.
The idea relies on the fact that the Dolbeault cohomology is computed by looking at vertical arrows in a double complex and its conjugate by looking at horizontal arrows.
Nevertheless the Bott-Chern and the Aeppli cohomologies compute the number of
ingoing and outgoing corners therefore, by combinatoric arguments, one gets that
the dimensions of the Bott-Chern and Aeppli cohomology groups are greater or equal
than the sum of Hodge numbers and their conjugates, which are 
greater or equal than the Betti numbers by Fr\"ohlicher.
As a corollary one gets also the stability of the $\del\delbar$-lemma under small
deformations of the complex structure
(see also \cite{voisin} and \cite{wu} for different proofs).
In \cite{angella-tomassini-algebraic} a generalization to double complexes is
developed, with applications to compact symplectic manifolds.\\

\begin{rem}\label{dimensione2}
Consider the special case when $X$ is a compact complex surface, i.e., $\dim_{\mathbb{C}}X=2$.
By duality the non-negative numbers $\Delta^1$ and $\Delta^2$ give
all the informations.
Since K\"ahlerness can be topologically characterized in terms of the parity of the first Betti number $b_1$, the K\"ahler condition is then
equivalent to the $\del\delbar$-lemma holding on $X$, leading to the equivalence:
$X$ is K\"ahler if and only if $\Delta^1=\Delta^2=0$.\\
Nevertheless we can be
even more precise, indeed, it is proven in \cite{teleman} that $\Delta^1$ vanishes on any compact complex surface
(see \cite{angella-dlousski-tomassini} for explicit examples). This is not true in higher dimension.
Therefore the number $\Delta^2$ measure the non-K\"ahlerness of a compact
surface:
$$
\mbox{K\"ahler}\qquad\iff\qquad\Delta^2=0.
$$
In general, on surfaces Teleman in \cite{teleman} proves that there are only two options for $\Delta^2$: it is either $0$ or $2$. For a generalization in higher dimension see \cite{angella-tomassini-verbitsky}.
\end{rem}

We have seen above that the Bott-Chern and the Aeppli numbers
dominate the Hodge numbers and then, by Fr\"ohlicher the Betti numbers.
In joint work with Angella in \cite{angella-tardini} (see also \cite{angella-survey}) we prove that they are also dominated by Hodge numbers.

\begin{theorem}[{\cite[Theorem 2.1, Remark 2.2]{angella-tardini}}]\label{complex-inequalities}
 Let $X$ be a compact complex manifold of complex dimension $n$. Then, for any $k\in\Z$,
 \begin{eqnarray*}
 \lefteqn{ \sum_{p+q=k} \dim_\mathbb{C} H^{p,q}_{A}(X) } \\[5pt]
 &\leq& \min\{k+1, (2n-k)+1\} \cdot \left( \sum_{p+q=k} \dim_\C H^{p,q}_{\overline\partial}(X) + \sum_{p+q=k+1} \dim_\C H^{p,q}_{\overline\partial}(X) \right) \\[5pt]
 &\leq& (n+1) \cdot \left( \sum_{p+q=k} \dim_\C H^{p,q}_{\overline\partial}(X) + \sum_{p+q=k+1} \dim_\C H^{p,q}_{\overline\partial}(X) \right) \;,
 \end{eqnarray*}
 and
 \begin{eqnarray*}
 \lefteqn{ \sum_{p+q=k} \dim_\C H^{p,q}_{BC}(X) } \\[5pt]
 &\leq& \min\{k+1, (2n-k)+1\} \cdot \left( \sum_{p+q=k} \dim_\C H^{p,q}_{\overline\partial}(X) + \sum_{p+q=k-1} \dim_\C H^{p,q}_{\overline\partial}(X) \right) \\[5pt]
 &\leq& (n+1) \cdot \left( \sum_{p+q=k} \dim_\C H^{p,q}_{\overline\partial}(X) + \sum_{p+q=k-1} \dim_\C H^{p,q}_{\overline\partial}(X) \right) \;.
 \end{eqnarray*}
\end{theorem}
\begin{proof}
The proof is essentially algebraic and, for example, the idea behind the first inequality is obtained by thinking that the outgoing corners in a zig-zag contribute to the Aeppli cohomology
and the extremal points of a zig-zag to the Dolbeault cohomology and/or its conjugate.
Therefore, for any outgoing corners we have two extremal points and the number
of outgoing corners depends on the length of the zig-zag.
For a detailed proof we refer to \cite{angella-tardini}
(see also \cite{angella-survey}).
\qed
\end{proof}
A similar result holds in case of double complexes
under some additional hypothesis of boundedness, leading to a similar
result in symplectic geometry.

\subsection{A characterization of the $\partial\overline\partial$-lemma}

By the above inequalities we then get that the difference $\sum_{p+q=k} \left( \dim_\C H^{p,q}_{BC}(X) - \dim_\C H^{p,q}_{A}(X) \right)$ is bounded from both above and below by the Hodge numbers. In \cite{angella-tardini} together with Angella we prove that there is
also a characterization of the $\partial\overline\partial$-Lemma in terms of this quantity.

\begin{theorem}[{\cite[Theorem 3.1]{angella-tardini}}]\label{characterization deldelbarlemma}
 A compact complex manifold $X$ satisfies the $\partial\overline\partial$-Lemma if and only if, for any $k\in\Z$, there holds
 $$ \sum_{p+q=k} \left( \dim_\C H^{p,q}_{BC}(X) - \dim_\C H^{p,q}_{A}(X) \right) \;=\; 0 \;. $$
\end{theorem}
\begin{proof}
The first implication is trivial. For the other one
notice that, roughly speaking, the vanishing of the numbers
$\sum_{p+q=k} \left( \dim_\C H^{p,q}_{BC}(X) - \dim_\C H^{p,q}_{A}(X) \right)$
means that the number of ingoing corners is equal to the number
of outgoing corners on any diagonal of the same total degree; since in degree $0$ we do not have ingoing corners then
we do not have any arrows in the picture of the double complex and therefore the $\del\delbar$-lemma
holds on $X$.
Nevertheless the precise proof of Theorem \ref{characterization deldelbarlemma} is based on Varouchas exact sequences \cite{varouchas} but it is not algebraic, indeed conjugation is needed; a similar result
cannot be expected in the symplectic case.
\qed
\end{proof}
\begin{rem}
This result means that on a compact complex manifold a non canonical isomorphism between the Bott-Chern and the Aeppli
cohomology forces all the natural maps in the cohomology diagram
to be isomorphisms and so these cohomologies are not providing additional informations on the manifold.
By the Schweitzer duality between the Bott-Chern and the Aeppli cohomology \cite[\S2.c]{schweitzer}, the above condition can be written just in terms of the Bott-Chern cohomology as follows: for any $k\in\Z$, there holds
$$ \sum_{p+q=k} \dim_\mathbb{C} H^{p,q}_{BC}(X) \;=\; \sum_{p+q=2n-k} \dim_\mathbb{C} H^{p,q}_{BC}(X) \;, $$
namely there is a symmetry in the Bott-Chern numbers.
The study of this property was initially motivated by the development of Sullivan theory of formality
in the context of Bott-Chern cohomology (see \cite{angella-tomassini-formality} and
\cite{tardini-tomassini-formality} for results in this direction).
\end{rem}

Notice that there exist special classes of complex manifolds where
the dimensions of the Bott-Chern (and by duality Aeppli) cohomology groups
can be computed explicitly by means of suitable sub-complexes of the complex of forms (see \cite{angella-kasuya-solvmanifolds})
making this result concrete in studying the $\del\delbar$-lemma.

\section{Symplectic cohomologies}

We consider now the symplectic case and we show that similar results hold
in this setting. 
Let $(X,\omega)$ be a compact symplectic manifold
of dimension $2n$, then Tseng and Yau in \cite{tsengyauI} define a symplectic version of the Bott-Chern and the Aeppli cohomology groups. Denoting with $A^\bullet(X)$ the space of differential forms on $X$,
the \emph{symplectic-$\star$-Hodge operator} (see \cite{brylinski})
$\star:A^\bullet(X)\longrightarrow A^{2n-\bullet }(X)$
is defined as follows: given $\beta\in A^k(X)$, for any $\alpha\in A^k(X)$ there
holds $\alpha\wedge\star\beta\;=\;(\omega^{-1})^k(\alpha,\beta)\;\omega^n$,
where on simple elements
$(\omega^{-1})^k(\alpha^1\wedge\ldots\wedge\alpha^k,\beta^1
\wedge\ldots\wedge\beta^k):=\det\left(\omega^{-1}(\alpha^i,\beta^j)\right)_{i,j}$.\\
The \emph{Brylinski co-differential} is defined as
\[
d^\Lambda:=\left[d,\Lambda\right]=d\Lambda-\Lambda d=(-1)^{k+1}\star d\star\;,
\]
where
$\Lambda:A^\bullet(X)\longrightarrow A^{\bullet -2}(X)$ is the adjoint of the
Lefschetz operator $L=\omega\wedge-:A^\bullet(X)\longrightarrow A^{\bullet +2}(X)$.
By definition $d^\Lambda:A^\bullet(X)\longrightarrow A^{\bullet -1}(X)$ and
the following relations hold: $\left(d^\Lambda\right)^2=0$ and 
$dd^\Lambda +d^\Lambda d=0$.\\
Notice that the operator $dd^\Lambda +d^\Lambda d$ is not the analogue
of the de-Rham Laplacian in the classical Riemannian Hodge theory because
it is not elliptic (it is always zero!) and we should think at $d^\Lambda$
as the analogue of the operator $d^c$ in complex geometry (actually
they are deeply related once fixed a compatible triple, see
\cite{tsengyauI} for more details).\\
Then, for $k\in\mathbb{Z}$, (see \cite{tsengyauI}) the \emph{$d^\Lambda$-cohomology groups} are
\[
H^k_{d^\Lambda}\left(X\right)
:=\frac{Ker(d^\Lambda)\cap A^k(X)}{Im \;d^\Lambda\cap A^k(X)},
\]
the \emph{symplectic Bott-Chern cohomology groups} are
\[
H^k_{d+d^\Lambda}\left(X\right)
:=\frac{Ker(d+d^\Lambda)\cap A^k(X)}{Im\; dd^\Lambda\cap A^k(X)}
\]
and the \emph{symplectic Aeppli cohomology groups} are
\[
H^k_{dd^\Lambda}\left(X\right)
:=\frac{Ker(dd^\Lambda)\cap A^k(X)}{\left(Im\; d+Im\; d^\Lambda\right)\cap A^k(X)}.
\]
By construction they are invariant under symplectomorphisms and so they are
good symplectic cohomologies encoding global invariants. For similar definitions
in the locally conformal symplectic setting see \cite{angella-otiman-tardini}.\\
Moreover these cohomology groups have been introduced because in
symplectic geometry the de Rham cohomology is not the appropriate
one when talking about Hodge theory.\\
Consider a compatible triple $(\omega,J,g)$ on $X$, namely
\begin{itemize}
\item{$J$ is a \emph{$\omega$-compatible almost-complex structure}, i.e.,
\begin{itemize}
\item{$\omega$ is positive on the $J$-complex lines, $\omega(\cdot, J\cdot)>0$;}
\item{$\omega$ is $J$-invariant, $\omega(J\cdot,J\cdot\cdot)=\omega(\cdot,\cdot\cdot)$;}
\end{itemize}}
\item{$g$ is the \emph{corresponding Riemannian metric} on $X$ defined by $g(\cdot,\cdot\cdot):=\omega(\cdot,J\cdot\cdot)$.}
\end{itemize} 
Denoting with $*$ the standard \emph{Hodge-operator}
with respect to the Riemannian metric $g$,
there are canonical isomorphisms (see \cite{tsengyauI})
\[
\mathcal{H}^k_{d^\Lambda}\left(X\right):=\ker\Delta_{d^\Lambda}
\simeq H^k_{d^\Lambda}\left(X\right),
\]
where $\Delta_{d^{\Lambda}} :=  d^{\Lambda*}d^\Lambda+d^\Lambda d^{\Lambda*}$ is a second-order elliptic self-adjoint
differential operator and
\[
\mathcal{H}^k_{d+d^\Lambda}\left(X\right):=\ker\Delta_{d+d^\Lambda}
\simeq H^k_{d+d^\Lambda}\left(X\right),\qquad
\mathcal{H}^k_{dd^\Lambda}\left(X\right):=\ker\Delta_{dd^\Lambda}
\simeq H^k_{dd^\Lambda}\left(X\right).
\]
where
$\Delta_{d+d^{\Lambda}}$, $\Delta_{dd^{\Lambda}}$ are fourth-order elliptic self-adjoint
differential operators defined by
\[
\begin{array}{lcl}
\Delta_{d+d^{\Lambda}}& := &(dd^{\Lambda})(dd{^\Lambda})^*+(dd^{\Lambda})^*(dd^{\Lambda})+
d^*d^{\Lambda} d^{\Lambda *}d+d^{\Lambda *}d d^*d^{\Lambda}+d^*d+d^{\Lambda *}d^{\Lambda},\\[10pt]
\Delta_{dd^{\Lambda}} & :=& (dd^{\Lambda})(dd{^\Lambda})^*+(dd^{\Lambda})^*(dd^{\Lambda})+
dd^{\Lambda *}d^\Lambda d^*+d^\Lambda d^*dd^{\Lambda *}+dd^*
+d^\Lambda d^{\Lambda *}.
\end{array}
\]
In particular, the symplectic cohomology groups are finite-dimensional vector spaces on a compact
symplectic manifold. For
$\sharp\in\left\{d^\Lambda,d+d^\Lambda,
dd^\Lambda\right\}$ we set $h^\bullet_\sharp:=:h^\bullet_\sharp(X):= \dim H^\bullet_\sharp(X)<\infty$ when the manifold $X$ is understood.\\
Similarly to the classical Hodge theory the differential forms closed both for the operators
$d$ and $d^\Lambda$ were called by Brylinski
\emph{symplectic harmonic} (\cite{brylinski}).
The existence of a symplectic harmonic form in each de Rham
cohomology class does not occur in general. As regards uniqueness there is no
hope, indeed on any symplectic manifold $(X,\omega)$ if $\alpha\in A^1(X)$
is symplectic-harmonic then $\alpha+df$ is still symplectic-harmonic, for any
smooth function $f$ on $X$, because $d^\Lambda(\alpha+df)=
d^\Lambda df=d\Lambda df=0$ for degree reasons.\\
In particular, the following facts are equivalent on a compact symplectic
manifold $(X^{2n},\omega)$ (cf. \cite{brylinski}, \cite{mathieu}, \cite{merkulov},
\cite{yan}, \cite{cavalcanti})
\begin{itemize}
\item{} the \emph{hard-Lefschetz condition} (HLC for short) holds, i.e., the maps
\[
L^k:H^{n-k}_{dR}(X)\longrightarrow H^{n+k}_{dR}(X), \qquad 0\leq k\leq n
\]
are all isomorphisms;
\item{} the \emph{Brylinski conjecture}, i.e., the existence of a
symplectic harmonic form in each de Rham cohomology class;
\item{} the \emph{$dd^\Lambda$-lemma}, i.e., every $d^\Lambda$-closed, $d$-exact form
is also $dd^\Lambda$-exact;
\item{} the natural maps induced by the identity $H^{\bullet}_{d+d^\Lambda}(X)
\longrightarrow H^{\bullet}_{dR}(X)$ are injective;
\item{} the natural maps induced by the identity $H^{\bullet}_{d+d^\Lambda}(X)
\longrightarrow H^{\bullet}_{dR}(X)$ are surjective;
\item{} the natural maps induced by the identity in the following diagram are isomorphisms
$$ \xymatrix{
  & H^{\bullet}_{d+d^\Lambda}(X) \ar[ld]\ar[rd] & \\
  H^{\bullet}_{dR}(X) \ar[rd] &  & H^{\bullet}_{d^\Lambda}(X). \ar[ld] \\
  & {\phantom{\;.}} H^{\bullet}_{dd^\Lambda}(X) \; &
} $$
\end{itemize}
In this sense $H^{\bullet}_{d+d^\Lambda}\left(X\right)$ and
$H^{\bullet}_{dd^\Lambda}\left(X\right)$ represent more appropriate cohomologies
talking about existence and uniqueness of harmonic representatives
on symplectic manifolds.\\
Nevertheless, in general, on a symplectic manifold of dimension $2n$ the following maps
are all isomorphisms (see \cite[Prop. 3.24]{tsengyauI})
$$\xymatrixcolsep{5pc}\xymatrix{
\mathcal{H}^{k}_{d+d^\Lambda}\left(X\right) \ar[d]^{L^{n-k}} \ar[r]^{*} &
\mathcal{H}^{2n-k}_{dd^\Lambda}\left(X\right)\ar[d]^{\Lambda^{n-k}}\\
\mathcal{H}^{2n-k}_{d+d^\Lambda}\left(X\right) \ar[r]^{*}   &
\mathcal{H}^{k}_{dd^\Lambda}\left(X\right),
}
$$
in particular, it follows that $h^k_{d+d^\Lambda}=h^{2n-k}_{d+d^\Lambda}=
h^k_{dd^\Lambda}=h^{2n-k}_{dd^\Lambda}$ for all $k=0,\ldots, 2n$.

\begin{rem}
Note that, as proved in \cite{angella-tardini} (see Theorem \ref{characterization deldelbarlemma}), on a compact complex manifold
the equality between the dimensions of the Bott-Chern cohomology groups
and the Aeppli cohomology groups characterizes the $\del\delbar$-lemma;
nevertheless, the "analogous" condition on a compact symplectic manifold $X$, namely
$h^{\bullet}_{d+d^\Lambda}(X)=
h^{\bullet}_{dd^\Lambda}(X)$, is
always verified.
\end{rem}

\subsection{Inequalities on compact symplectic manifolds}

The proof of Theorem \ref{complex-inequalities} is essentially algebraic and it can
be generalized to double complexes with some hypothesis of boundedness.
For the general statement we refer to \cite{angella-tardini}, here we consider
the application to the symplectic cohomologies.
Let $X$ be a compact manifold of dimension $2n$ endowed with a symplectic structure $\omega$.
As in \cite{brylinski, cavalcanti}, we define the double complex associated to
$(A^\bullet(X),d,d^\Lambda)$ as
$$ \left( B^{\bullet_1,\bullet_2} \;:=\; \wedge^{\bullet_1-\bullet_2}X \otimes \beta^{\bullet_2} ,\; d\otimes\mathrm{id},\; d^\Lambda\otimes\beta \right) \;, $$
where $\beta$ is a generator of the infinite cyclic commutative group $\beta^\Z$. Note that, for any $q\in\Z$, we have
$$ B^{p,q} \;=\; \{0\}, \qquad p\not\in\{q, \ldots, q+2n \} \;, $$
hence there exists a diagonal strip of width $2n+1$ such that the double complex $B^{\bullet,\bullet}$ has support in this strip. 
In the picture below we have an example for $2n=4$.\\

\begin{center}
\begin{tikzpicture}
\newcommand\un{0.9}

\draw[help lines, step=\un] (0,0) grid (7*\un,7*\un);

\foreach \x in {-1,...,5}
  \node at (\un+\un*.5+\un*\x,-.3) {\x};
\foreach \y in {-2,...,4}
  \node at (-.3,\un*2+\un*.5+\un*\y) {\y};

\coordinate (AA) at (0*\un, 1*\un);
\coordinate (A) at (0*\un, 2*\un);
\coordinate (B) at (1*\un, 2*\un);
\coordinate (C) at (1*\un, 3*\un);
\coordinate (D) at (2*\un, 3*\un);
\coordinate (E) at (2*\un, 4*\un);
\coordinate (F) at (3*\un, 4*\un);
\coordinate (G) at (3*\un, 5*\un);
\coordinate (H) at (4*\un, 5*\un);
\coordinate (I) at (4*\un, 6*\un);
\coordinate (II) at (5*\un, 6*\un);
\coordinate (BB) at (5*\un, 7*\un);
\coordinate (CC) at (6*\un, 7*\un);

\coordinate (jj) at (3*\un, 0*\un);
\coordinate (j) at (4*\un, 0*\un);
\coordinate (k) at (4*\un, 1*\un);
\coordinate (l) at (5*\un, 1*\un);
\coordinate (m) at (5*\un, 2*\un);
\coordinate (n) at (6*\un, 2*\un);
\coordinate (o) at (6*\un, 3*\un);
\coordinate (p) at (7*\un, 3*\un);
\coordinate (q) at (7*\un, 4*\un);

\coordinate (r) at (4*\un+1/2*\un, 4*\un+1/2*\un);
\coordinate (s) at (4*\un+1/2*\un+1/8*\un, 4*\un+1/2*\un+1/8*\un);
\coordinate (t) at (4*\un+1/2*\un-1/8*\un, 4*\un+1/2*\un-1/8*\un);

\coordinate (u) at (0*\un+1/2*\un, 0*\un+1/2*\un);
\coordinate (v) at (0*\un+1/2*\un+1/8*\un, 0*\un+1/2*\un+1/8*\un);
\coordinate (z) at (0*\un+1/2*\un-1/8*\un, 0*\un+1/2*\un-1/8*\un);

\newcommand{\raggio}{1*\un pt}
\fill (r) circle (\raggio);
\fill (s) circle (\raggio);
\fill (t) circle (\raggio);
\fill (u) circle (\raggio);
\fill (v) circle (\raggio);
\fill (z) circle (\raggio);

\fill [gray!15] (0+1/50*\un,2*\un+1/50*\un) -- (0+1/50*\un,3*\un-1/50*\un) -- (1*\un-1/50*\un,3*\un-1/50*\un) -- (1*\un-1/50*\un,2*\un+1/50*\un);
\fill [gray!15] (0+1/50*\un,3*\un+1/50*\un) -- (0+1/50*\un,4*\un-1/50*\un) -- (1*\un-1/20*\un,4*\un-1/50*\un) -- (1*\un-1/20*\un,3*\un+1/50*\un);
\fill [gray!15] (0+1/50*\un,4*\un+1/50*\un) -- (0+1/50*\un,5*\un-1/50*\un) -- (1*\un-1/50*\un,5*\un-1/50*\un) -- (1*\un-1/50*\un,4*\un+1/50*\un);
\fill [gray!15] (0+1/50*\un,5*\un+1/50*\un) -- (0+1/50*\un,6*\un-1/50*\un) -- (1*\un-1/50*\un,6*\un-1/50*\un) -- (1*\un-1/50*\un,5*\un+1/50*\un);
\fill [gray!15] (0+1/50*\un,6*\un+1/50*\un) -- (0+1/50*\un,7*\un-1/50*\un) -- (1*\un-1/20*\un,7*\un-1/50*\un) -- (1*\un-1/20*\un,6*\un+1/50*\un);

\fill [gray!15] (1*\un+1/50*\un,3*\un+1/50*\un) -- (1*\un+1/50*\un,4*\un-1/50*\un) -- (2*\un-1/20*\un,4*\un-1/50*\un) -- (2*\un-1/20*\un,3*\un+1/50*\un);
\fill [gray!15] (1*\un+1/50*\un,4*\un+1/50*\un) -- (1*\un+1/50*\un,5*\un-1/50*\un) -- (2*\un-1/20*\un,5*\un-1/50*\un) -- (2*\un-1/20*\un,4*\un+1/50*\un);
\fill [gray!15] (1*\un+1/50*\un,5*\un+1/50*\un) -- (1*\un+1/50*\un,6*\un-1/50*\un) -- (2*\un-1/20*\un,6*\un-1/50*\un) -- (2*\un-1/20*\un,5*\un+1/50*\un);
\fill [gray!15] (1*\un+1/50*\un,6*\un+1/50*\un) -- (1*\un+1/50*\un,7*\un-1/50*\un) -- (2*\un-1/20*\un,7*\un-1/50*\un) -- (2*\un-1/20*\un,6*\un+1/50*\un);

\fill [gray!15] (2*\un+1/50*\un,4*\un+1/50*\un) -- (2*\un+1/50*\un,5*\un-1/50*\un) -- (3*\un-1/20*\un,5*\un-1/50*\un) -- (3*\un-1/20*\un,4*\un+1/50*\un);
\fill [gray!15] (2*\un+1/50*\un,5*\un+1/50*\un) -- (2*\un+1/50*\un,6*\un-1/50*\un) -- (3*\un-1/20*\un,6*\un-1/50*\un) -- (3*\un-1/20*\un,5*\un+1/50*\un);
\fill [gray!15] (2*\un+1/50*\un,6*\un+1/50*\un) -- (2*\un+1/50*\un,7*\un-1/50*\un) -- (3*\un-1/20*\un,7*\un-1/50*\un) -- (3*\un-1/20*\un,6*\un+1/50*\un);

\fill [gray!15] (3*\un+1/50*\un,5*\un+1/50*\un) -- (3*\un+1/50*\un,6*\un-1/50*\un) -- (4*\un-1/20*\un,6*\un-1/50*\un) -- (4*\un-1/20*\un,5*\un+1/50*\un);
\fill [gray!15] (3*\un+1/50*\un,6*\un+1/50*\un) -- (3*\un+1/50*\un,7*\un-1/50*\un) -- (4*\un-1/20*\un,7*\un-1/50*\un) -- (4*\un-1/20*\un,6*\un+1/50*\un);

\fill [gray!15] (4*\un+1/50*\un,6*\un+1/50*\un) -- (4*\un+1/50*\un,7*\un-1/50*\un) -- (5*\un-1/20*\un,7*\un-1/50*\un) -- (5*\un-1/20*\un,6*\un+1/50*\un);

\fill [gray!15] (4*\un+1/50*\un,0*\un+1/50*\un) -- (4*\un+1/50*\un,1*\un-1/50*\un) -- (5*\un-1/20*\un,1*\un-1/50*\un) -- (5*\un-1/20*\un,0*\un+1/50*\un);
\fill [gray!15] (5*\un+1/50*\un,0*\un+1/50*\un) -- (5*\un+1/50*\un,1*\un-1/50*\un) -- (6*\un-1/20*\un,1*\un-1/50*\un) -- (6*\un-1/20*\un,0*\un+1/50*\un);
\fill [gray!15] (6*\un+1/50*\un,0*\un+1/50*\un) -- (6*\un+1/50*\un,1*\un-1/50*\un) -- (7*\un-1/20*\un,1*\un-1/50*\un) -- (7*\un-1/20*\un,0*\un+1/50*\un);

\fill [gray!15] (5*\un+1/50*\un,1*\un+1/50*\un) -- (5*\un+1/50*\un,2*\un-1/50*\un) -- (6*\un-1/20*\un,2*\un-1/50*\un) -- (6*\un-1/20*\un,1*\un+1/50*\un);
\fill [gray!15] (6*\un+1/50*\un,1*\un+1/50*\un) -- (6*\un+1/50*\un,2*\un-1/50*\un) -- (7*\un-1/20*\un,2*\un-1/50*\un) -- (7*\un-1/20*\un,1*\un+1/50*\un);

\fill [gray!15] (6*\un+1/50*\un,2*\un+1/50*\un) -- (6*\un+1/50*\un,3*\un-1/50*\un) -- (7*\un-1/20*\un,3*\un-1/50*\un) -- (7*\un-1/20*\un,2*\un+1/50*\un);

\draw (AA) -- (A) -- (B) -- (C) -- (D) -- (E) -- (F) -- (G) -- (H) -- (I) -- (II) -- (BB) -- (CC);
\draw (jj) -- (j) -- (k) -- (l) -- (m) -- (n) -- (o) -- (p) -- (q);

\begingroup\makeatletter\def\f@size{6}\check@mathfonts
\node at (1*\un+1/2*\un-.1+.1, 2*\un+1/2*\un) {$\Lambda^0\otimes\beta^0$};
\node at (2*\un+1/2*\un-.1+.1, 2*\un+1/2*\un) {$\Lambda^1\otimes\beta^0$};
\node at (3*\un+1/2*\un-.1+.1, 2*\un+1/2*\un) {$\Lambda^2\otimes\beta^0$};
\node at (4*\un+1/2*\un-.1+.1, 2*\un+1/2*\un) {$\Lambda^3\otimes\beta^0$};
\node at (5*\un+1/2*\un-.1+.1, 2*\un+1/2*\un) {$\Lambda^4\otimes\beta^0$};

\node at (2*\un+1/2*\un-.1+.1, 3*\un+1/2*\un) {$\Lambda^0\otimes\beta^1$};
\node at (3*\un+1/2*\un-.1+.1, 3*\un+1/2*\un) {$\Lambda^1\otimes\beta^1$};
\node at (4*\un+1/2*\un-.1+.1, 3*\un+1/2*\un) {$\Lambda^2\otimes\beta^1$};
\node at (5*\un+1/2*\un-.1+.1, 3*\un+1/2*\un) {$\Lambda^3\otimes\beta^1$};
\node at (6*\un+1/2*\un-.1+.1, 3*\un+1/2*\un) {$\Lambda^4\otimes\beta^1$};

\node at (0*\un+1/2*\un-.1+.1, 1*\un+1/2*\un) {$\Lambda^0\otimes\beta^{-1}$};
\node at (1*\un+1/2*\un-.1+.1, 1*\un+1/2*\un) {$\Lambda^1\otimes\beta^{-1}$};
\node at (2*\un+1/2*\un-.1+.1, 1*\un+1/2*\un) {$\Lambda^2\otimes\beta^{-1}$};
\node at (3*\un+1/2*\un-.1+.1, 1*\un+1/2*\un) {$\Lambda^3\otimes\beta^{-1}$};
\node at (4*\un+1/2*\un-.1+.1, 1*\un+1/2*\un) {$\Lambda^4\otimes\beta^{-1}$};

\endgroup

\end{tikzpicture}
\end{center}

\medskip

The Bott-Chern and Aeppli cohomologies of $B^{\bullet,\bullet}$ are related to the symplectic cohomologies of $X$,
$H^{\bullet}_{d+d^\Lambda}(X)\; ,\; H^{\bullet}_{dd^\Lambda}(X)$,
more precisely,
$$
H^{\bullet_1,\bullet_2}_{BC}(B^{\bullet,\bullet}) \;=\; H^{\bullet_1-\bullet_2}_{d+d^\Lambda}(X) \otimes \beta^{\bullet_2},
\qquad
H^{\bullet_1,\bullet_2}_{A}(B^{\bullet,\bullet}) \;=\; H^{\bullet_1-\bullet_2}_{dd^\Lambda}(X) \otimes \beta^{\bullet_2} \;.
$$
The conjugate-Dolbeault and Dolbeault cohomologies of $B^{\bullet,\bullet}$ are both related to the de Rham cohomology of $X$.
With the same idea of the proof of Theorem \ref{complex-inequalities} we can prove
the following

\begin{theorem}[{\cite[Theorem 6.2]{angella-tardini}}]\label{thm:upper-bound-symplectic}
 Let $X$ be a compact differentiable manifold of dimension $2n$ endowed with a symplectic structure $\omega$. Then, for any $k\in\Z/2\Z$,
 $$
 \sum_{h = k \,\mathrm{mod}\, 2} \dim_\rr H^{h}_{d+d^\Lambda}(X)
 \;\leq\;
 2(2n+1) \cdot \sum_{h\in\Z} \dim_\rr H^{h}_{dR}(X;\rr) \;,
 $$
 and
 $$
 \sum_{h = k \,\mathrm{mod}\, 2} \dim_\rr H^{h}_{dd^\Lambda}(X)
 \;\leq\;
 2(2n+1) \cdot \sum_{h\in\Z} \dim_\rr H^{h}_{dR}(X;\rr) \;.
 $$
\end{theorem}

\subsection{A characterization of the Hard Lefschetz condition}

In \cite{angella-tomassini-algebraic} Angella and Tomassini, starting from a purely
algebraic point of view,
introduce on a compact symplectic manifold $(X^{2n},\omega)$
the following non-negative integers 
\[
\Delta^k:= h^k_{d+d^\Lambda}+h^k_{dd^\Lambda}-2b_k\geq 0,
\qquad k\in\mathbb{Z},
\]
proving that, similarly to the complex case, their vanishing characterizes the $dd^\Lambda$-lemma which is equivalent to the validity
of the Hard-Lefschetz condition.
In this sense these numbers measure the HLC-degree of a symplectic manifold,
as their analogue in the complex case do (cf. \cite{angella-tomassini-inequality}).\\
Now, as already observed by Chan and Suen in \cite{chansuen},
using the equality $\dim H^\bullet_{d+d^\Lambda}(X)=
\dim H^\bullet_{dd^\Lambda}(X)$ proved in \cite{tsengyauI}, we get
\[
\Delta^k=2(h^k_{d+d^\Lambda}-b_k),\qquad k\in\mathbb{Z};
\]
therefore we can simplify them as in \cite{tardini-tomassini-symplectic}, considering just the difference between the dimensions of the
Bott-Chern and the de Rham cohomology groups. We define
\[
\tilde\Delta^k:=h^k_{d+d^\Lambda}-b_k,\qquad k\in\mathbb{Z}.
\]
Notice that a similar simplification can not be done in the complex case (cf. \cite{schweitzer}).
We put in evidence that, by duality, $\tilde\Delta^k=\tilde\Delta^{2n-k}$,
$k=0,\ldots 2n$, so for a compact symplectic manifold $(X,\omega)$
of dimension $2n$ we will refer
to $\tilde\Delta^k$, $k=0\ldots n$, as the \emph{non-HLC-degrees} of $X$.
Note that $\tilde\Delta^0=0$.\\
As a consequence of the positivity of $\Delta^k$, for any $k$, we have that
for all $k=1,\ldots,n$
\[
b_k\leq h^k_{d+d^\Lambda}
\]
on a compact symplectic $2n$-dimensional manifold.\\
Moreover the equalities 
\[
b_k=h^k_{d+d^\Lambda}, \qquad \forall k=1,\ldots,n,
\]
hold on a compact symplectic $2n$-dimensional manifold if and only if it 
satisfies the Hard-Lefschetz condition; namely 
the equality $b_\bullet= h^\bullet_{d+d^\Lambda}$ ensures
the bijectivity of the natural maps $H^\bullet_{d+d^\Lambda}(X)\longrightarrow
H^\bullet_{dR}(X)$, and hence the $dd^\Lambda$-lemma.\\
This considerations can be inserted in the more general setting of generalized complex manifolds,
see \cite{chansuen} for more details.\\
Similarly to the complex case where $\Delta^2$ characterizes the K\"ahlerianity
of a compact complex surface, if $2n=4$ we want
to show that the only degree which characterizes the Hard Lefschetz Condition is $\tilde\Delta^2$.
Notice that, differently to the complex case, in any dimension we have the following 
\begin{theorem}[{\cite[Theorem 4.3]{tardini-tomassini-symplectic}}]
Let $(X^{2n},\omega)$ be a compact symplectic manifold, then
the natural map induced
by the identity
\[
H^1_{d+d^\Lambda}(X)\longrightarrow H^1_{dR}(X)
\]
is an isomorphism. In particular,
$$\tilde\Delta^1=0.$$
\end{theorem}
\begin{proof}
For the sake of completeness we briefly recall here the proof. 
For the surjectivity, if $\alpha$ is
a $d$-closed $1$-form, then it is also $d^\Lambda$-closed, indeed
$$
d^\Lambda\alpha=\left[d,\Lambda\right]\alpha=-\Lambda d\alpha=0.
$$
We need to prove the injectivity. Let $a=[\alpha]\in H^1_{d+d^\Lambda}(X)$ be such that $a=0$ in $H^1_{dR}(X)$, namely
$\alpha=df$ for some smooth function $f$ on $X$.
Considering the Hodge decomposition of $f$ with respect to the $d^\Lambda$-cohomology (cf. \cite{tsengyauI})
we get $f=c+d^\Lambda\beta$ with $c$ constant and $\beta$ differential
$1$-form. Hence
\[
\alpha=df=d(c+d^\Lambda\beta)=dd^\Lambda\beta,
\]
i.e., $[\alpha]=0\in H^1_{d+d^\Lambda}(X)$.\\
As a consequence, $b_1=h^1_{d+d^\Lambda}$, implying
$\tilde\Delta^1=h^1_{d+d^\Lambda}-b_1=0$ and concluding the proof.
\qed
\end{proof}
The analog result for the complex Bott-Chern cohomology is not true,
see e.g., \cite[Remark 3.6]{angella-kasuya-solvmanifolds}.
The previous Theorem lead us to the following quantitative characterization of the Hard Lefschetz condition in dimension $4$.
\begin{theorem}[{\cite[Theorem 4.5]{tardini-tomassini-symplectic}}]\label{characterization HLC}
Let $(X^4,\omega)$ be a compact symplectic $4$-manifold, then
it satisfies
\[
HLC \iff \tilde\Delta^2=0 \iff b_2=h^2_{d+d^\Lambda}.
\]
\end{theorem}
Therefore in $4$-dimensions it is possible to study the Hard Lefschetz condition
by studying the dependence of the space $H^2_{d+d^\Lambda}(X)$
on the symplectic structure.\\

\begin{rem}
As shown in \cite{teleman} on a compact complex surface $\Delta^2\in\left\lbrace 0,2\right\rbrace$;
in \cite{tardini-tomassini-symplectic} with Tomassini we provide an 
explicit example of a compact symplectic $4$-manifold with
$\Delta^2\notin\left\lbrace 0,2\right\rbrace$, or equivalently $\tilde\Delta^2\notin\left\lbrace 0,1\right\rbrace$, showing hence a different behavior in the symplectic case.
More precisely we compute the non-HLC degree $\tilde\Delta^2$ when $X$ is a compact
$4$-dimensional manifold diffeomorphic to
a solvmanifold $\Gamma\backslash G$ (i.e., the compact quotient of a connected
simply-connected solvable Lie group $G$ by a discrete cocompact subgroup
$\Gamma$) admitting a
left-invariant symplectic structure; for a partial computation cfr.
\cite[Table 2]{angellakasuya}.\\
In detail,
if $X=\Gamma\backslash G$ is a compact solvmanifold of dimension $4$ with
$\omega$
left-invariant symplectic structure,
then, according to $\mathfrak{g}=Lie(G)$, we have the following cases
\begin{itemize}
\item[a)] if $\mathfrak{g}=\mathfrak{g}_{3,1}\oplus\mathfrak{g}_1$, then
$\tilde\Delta^2=1$;
\item[b)] if $\mathfrak{g}=\mathfrak{g}_{1}\oplus\mathfrak{g}_{3,4}^{-1}$, then
$\tilde\Delta^2=0$;
\item[c)] if $\mathfrak{g}=\mathfrak{g}_{4,1}$, then
$\tilde\Delta^2=2$.
\end{itemize}
See \cite{tardini-tomassini-symplectic} for the computations.\\
Notice that, by applying Theorem \ref{thm:upper-bound-symplectic}, with an easy computation we obtain the following (quite large) inequalities for a general compact symplectic
$4$-manifold $(X^4,\omega)$,
$$
b_2\leq h^2_{d+d^\Lambda}\leq 10\;b_2+20\;b_1+18
$$
and
$$
0\leq \tilde\Delta^2\leq 9\;b_2+20\;b_1+18.
$$
\end{rem}


\begin{thebibliography}{99}


\bibitem{aeppli}
Aeppli, A.: On the cohomology structure of Stein manifolds, in {\em Proc. Conf. Complex Analysis (Minneapolis, Minn., 1964)}, Springer, Berlin, 58--70 (1965)

\bibitem{angella-iwasawa}
Angella, D.: The cohomologies of the Iwasawa manifold and of its small deformations. J. Geom. Anal. \textbf{23}, no. 3, 1355--1378 (2013)

\bibitem{angella-survey}
Angella, D.: On the Bott-Chern and Aeppli cohomology. In {\em Bielefeld Geometry \& Topology Days} (2015)


\bibitem{angella-kasuya-solvmanifolds} Angella, D., Kasuya, H.: Bott-Chern cohomology of solvmanifolds \texttt{arXiv:1212.5708 [math.DG]}, (2016)

\bibitem{angellakasuya} Angella, D., Kasuya, H.: 
Symplectic Bott-Chern cohomology of solvmanifolds,
 \texttt{arXiv:1308.4258 [math.SG]}, (2016)

\bibitem{angella-tardini} Angella, D., Tardini, N.: 
Quantitative and qualitative cohomological properties for non-K\"ahler
manifolds. Proc. Amer. Math. Soc. \textbf{145} no.~1, 273--285 (2017)

\bibitem{angella-tomassini-inequality}
Angella, D. , Tomassini, A.: On the $\partial\overline{\partial}$-lemma and Bott-Chern cohomology. Invent. Math. \textbf{192} no.~1, 71--81 (2013)

\bibitem{angella-tomassini-algebraic}
Angella, D., Tomassini, A.: Inequalities \`a la Fr\"olicher and cohomological decompositions. J. Noncommut. Geom. \textbf{9} no.~2, 505--542 (2015)

\bibitem{angella-tomassini-formality}
Angella, D., Tomassini, A.: On Bott-Chern cohomology and formality. J. Geom.
Phys. \textbf{93}, 52--61 (2015)

\bibitem{angella-dlousski-tomassini}
Angella, D., Dloussky, G., Tomassini, A.:
On Bott-Chern cohomology of compact complex surfaces. Ann. Mat. Pura Appl.\textbf{195} no.~1, 199--217 (2016)


\bibitem{angella-tomassini-verbitsky} Angella, D., Tomassini, A., Verbitsky,
M.:
On non-K\"ahler degrees of complex manifolds, \texttt{arXiv:1605.03368},
(2016)


\bibitem{angella-otiman-tardini} Angella, D., Otiman, A., 
Tardini, N.: Cohomological properties of locally conformally symplectic structures, in preparation.

\bibitem{bismut}
Bismut, J.-M.:  A local index theorem for non-K\"ahler manifolds. Math. Ann.
\textbf{284}, 681--699 (1989)

\bibitem{bott-chern}
Bott, R., Chern, S.~S.: Hermitian vector bundles and the equidistribution of the zeroes of their holomorphic sections. Acta Math. \textbf{114} no.~1, 71--112  (1965)


\bibitem{brylinski} Brylinski, J.-L.: A differential complex for Poisson manifolds,
J. Differ. Geom. \textbf{28} no.~1, 93--114 (1988)

\bibitem{chansuen} Chan, K., Suen, Y.-H.: A Fr\"olicher-type inequality for generalized complex manifolds. Ann. Global Anal. Geom. \textbf{47} no.~2, 135--145 (2015),

\bibitem{cavalcanti} Cavalcanti, G. R.: New aspects of the $dd^c$-lemma,
Oxford University D. Phil thesis, \texttt{arXiv:math/0501406v1 [math.DG]},
(2005)

\bibitem{deligne-griffiths-morgan-sullivan}
Deligne, P., Griffiths, Ph.~A., Morgan, J., Sullivan, D.~P.: Real homotopy theory of K\"ahler manifolds. Invent. Math. \textbf{29} no.~3, 245--274 (1975)

\bibitem{frolicher}
Fr\"olicher, A.: Relations between the cohomology groups of Dolbeault and topological invariants. Proc. Natl. Acad. Sci. USA \textbf{41} no.~9, 641--644
(1955)


\bibitem{kodaira}
Kodaira, K.: On the structure of compact complex analytic surfaces. I. Am. J.
Math. \textbf{86}, 751--798 (1964)

\bibitem{kodaira-spencer}
Kodaira, K., Spencer, D.~C.: On deformations of complex analytic structures. III. Stability theorems for complex structures. Ann. of Math. (2) \textbf{71}, 43--76
(1960)

\bibitem{mathieu} Mathieu, O.: Harmonic cohomology classes of symplectic manifolds. Comment. Math. Helv. \textbf{70} no.~1, 1--9 (1995)

\bibitem{merkulov} Merkulov, S. A.: Formality of canonical symplectic complexes and Frobenius
manifolds. Int. Math. Res. Not. \textbf{1998} no.~14, 727--733  (1998)

\bibitem{michelsohn}
Michelsohn, M. L.: On the Existence of Special Metrics in Complex
Geometry. Acta Math. \textbf{143}, 261--295 (1983)

\bibitem{miyaoka}
Miyaoka, Y.: K\"ahler metrics on elliptic surfaces. Proc. Jpn. Acad.
\textbf{50} (8), 533--536 (1974)
 
\bibitem{schweitzer}
Schweitzer, M: Autour de la cohomologie de Bott-Chern, Pr\'epublication de l'Institut Fourier no.~703, \texttt{arXiv:0709.3528}, (2007)

\bibitem{siu}
Siu, Y. T.: Every K3 surface is K\"ahler. Invent. Math. \textbf{73} (1),
139--150 (1983)

\bibitem{tardini-tomassini-formality}
Tardini, N., Tomassini, A.: On geometric Bott-Chern formality and deformations, \emph{Ann. Mat. Pura Appl.} 
(2016), doi:10.1007/s10231-016-0575-6.  Article electronically published on May 02, 2016.

\bibitem{tardini-tomassini-symplectic} Tardini, N., Tomassini, A.:
On the cohomology of almost-complex and symplectic manifolds and proper surjective maps, Internat. J. Math. \textbf{27} no.~12, 1650103 (20 pages) (2016)

\bibitem{teleman} Teleman, A.: The pseudo-effective cone of a non-K\"ahlerian surface and applications, Math. Ann. \textbf{335}
965--989 (2006)

\bibitem{tsengyauI} Tseng, L.-S. , Yau, S.-T.: Cohomology and Hodge Theory on Symplectic manifolds: I. J. Differ. Geom. \textbf{91} no.~3, 383--416
 (2012)

\bibitem{varouchas}
Varouchas, J.: Propriet\'es cohomologiques d'une classe de vari\'et\'es analytiques complexes compactes, in {\em S\'eminaire d'analyse P. Lelong-P. Dolbeault-H. Skoda, ann\'ees 1983/1984}, Lecture Notes in Math., vol. \textbf{1198}, 233--243, Springer, Berlin (1986)

\bibitem{voisin} Voisin, C.: Th\'eorie de Hodge et g\'eom\'etrie alg\'ebrique complexe, Cours Sp\'ecialis\'es, \textbf{10} Soci\'et\'e Math\'ematique
de France, Paris, (2002)

\bibitem{wu} Wu, C.-C.: On the geometry of superstrings with torsion, Thesis (Ph.D.) Harvard University, Proquest LLC,
Ann Arbor, MI, (2006)

\bibitem{yan} Yan, D.: Hodge structure on symplectic manifolds. Adv. Math.\textbf{120} no.~1, 143--154 (1996)




\end{thebibliography}
\end{document}